\documentclass[12pt]{amsart}
\usepackage{amscd}
\usepackage{amsmath}
\usepackage{latexsym}
\usepackage{graphicx}
\usepackage{amsfonts}
\usepackage{amssymb}
\usepackage[all]{xy}
\DeclareMathOperator*{\colim}{colim}

\newcommand{\Z}{\mathbb{Z}}

\newcommand{\GW}{\mathbb{G}W}

\newcommand{\A}{\mathbb A}
\newcommand{\Gm}{\mathbb{G}_{m}}
\renewcommand{\L}{\mathbb{L}}

\newcommand{\cA}{\mathcal{A}}
\newcommand{\cB}{\mathcal{B}}
\newcommand{\cE}{\mathcal{E}}
\newcommand{\cH}{\mathcal{H}}
\newcommand{\scL}{\mathcal{L}}
\newcommand{\cO}{\mathcal{O}}
\newcommand{\cT}{\mathcal{T}}
\newcommand{\cU}{\mathcal{U}}
\newcommand{\sA}{\mathcal{A}}
\newcommand{\sH}{\mathcal{H}}
\newcommand{\sL}{\mathcal{L}}

\newcommand{\pt}{{\textrm{pt}}}

\newcommand{\susp}{S^1\!\wedge{\!}}
\newcommand{\Tate}{\widehat{\mathbf{H}}}
\newcommand{\tate}{\widehat{\mathrm{H}}}

\numberwithin{equation}{section}
\theoremstyle{plain}
\newtheorem{theorem}[equation]{Theorem}

\newtheorem{corollary}[equation]{Corollary}
\newtheorem{proposition}[equation]{Proposition}
\newtheorem{lemma}[equation]{Lemma}
\newtheorem{substuff}{\bf Remark}[equation] %for numbering

\theoremstyle{definition}

\newtheorem{example}[equation]{Example}

\theoremstyle{remark}
\newtheorem{remark}[equation]{Remark}

\newtheorem{subremark}[substuff]{Remark} %for numbering

\def\smap#1{\ {\buildrel #1 \over \rightarrow}\ }

\def\map#1{{\buildrel #1 \over \longrightarrow}}

\newcommand{\Spec}{\operatorname{Spec}}

\newcommand{\Hom}{\operatorname{Hom}}

\newcommand{\eps}{\varepsilon}
\newcommand{\K}{\mathbb{K}}
\newcommand{\ffi}{\varphi}

\newcommand{\on}{\hspace{1ex}\operatorname{on}\hspace{1ex}}
\renewcommand{\P}{\mathbb{P}}
\newcommand{\Proj}{\operatorname{Proj}}
\newcommand{\zx}{z_X}  %replaces $X\times t_0$
\newcommand{\sPerf}{\operatorname{sPerf}}
\newcommand{\Vect}{\operatorname{Vect}}

\begin{document}
\title[Witt groups of some singular schemes]
{Grothendieck-Witt groups of some singular schemes}
\date{\today}

\author{Max Karoubi}
\address{Universit\'e Denis Diderot Paris 7 \\
Institut Math\'ematique de Jussieu --- Paris Rive Gauche}
\email{max.karoubi@gmail.com}
\urladdr{http://webusers.imj-prg.fr/~max.karoubi}

\author{Marco Schlichting}
\address{Math.\ Institute, University of Warwick,Coventry CV4 7AL, UK}
\email{M.Schlichting@warwick.ac.uk}
\urladdr{http://http://homepages.warwick.ac.uk/~masiap}
\thanks{Schlichting was partially supported by a Leverhulme fellowship}
%not EPSRC grant EP/M001113/1}

\author{Charles Weibel}
\address{Math.\ Dept., Rutgers University, New Brunswick, NJ 08901, USA}
\email{weibel@math.rutgers.edu}\urladdr{http://math.rutgers.edu/~weibel}
\thanks{Weibel was supported by NSF grants}

\begin{abstract} 
We establish some structural results for the Witt 
and Grothendieck--Witt groups of schemes
over $\Z[1/2]$, including homotopy invariance 
for Witt groups and a formula for the
Witt and Grothendieck--Witt groups of punctured affine spaces over a scheme. 
All these results hold for singular schemes and at the level of spectra.
\end{abstract}
\maketitle

%\tableofcontents

\section{Introduction} 

Let $X$ be a quasi-projective scheme, or more generally a
scheme with an ample family of line bundles, such that 
$2\in\cO(X)^\times$. 
%$2 \in \Gamma(X,O_X)^*$.  
In this paper, we show how new techniques can help calculate Balmer's
4-periodic Witt groups $W^n(X)$ of $X$, in particular when $X$ is singular,
including the classical Witt group $W^0(X)$ of \cite{Knebusch}.  For example,
we establish homotopy invariance in Theorem \ref{thm:hi}: if $V\to X$ is a
vector bundle then $W^*(X)\cong W^*(V)$.  (If $X$ is affine, this was proven
by %Karoubi 
the first author in \cite[3.10]{MKlocalisation}; if $X$ is
regular, it was proven by Balmer \cite{Balmer} and Gille \cite{Gille}.)

The formula $W^0(X\times\Gm)\cong W^0(X) \oplus W^0(X)$,
which holds for nonsingular
schemes by a result of Balmer--Gille \cite{BalmerGille}, fails for 
curves with nodal singularities (see Example \ref{ex:node}), but holds if
$K_{-1}(X)=0$; see Theorem \ref{thm:LCn}. We show in {\em loc.\,cit.}\
that a similar result holds for the punctured affine space $X\times(\A^n-0)$ over $X$.

All of this holds at the spectrum level.  Recall from
\cite[7.1]{Schlichting.Fund} that there are spectra $L^{[r]}(X)$
whose homotopy groups are Balmer's 4-periodic triangular Witt groups:
$\pi_iL^{[r]}(X)=L_i^{[r]}(X) \cong W^{r-i}(X).$ 
%are the homotopy groups of a spectrum $L^{[r]}(X)$.
We write $L(X)$ for $L^{[0]}(X)$.
Our first main theorem %The following theorem 
shows that the functors $L^{[r]}$ are homotopy invariant.

\begin{theorem}\label{thm:hi}
Let $X$ be a scheme over $\Z[1/2]$, %$\Spec(\Z[1/2])$, 
with an ample family of line bundles.
  %such that $2\in\cO(X)^\times.$ 
If $V\to X$ is a vector bundle, %over $X$, 
then the projection induces a stable equivalence of $L$-theory spectra
$L(X) \smap{\simeq} L(V)$.
On homotopy groups, $W^n(X)\! \smap{\cong} W^n(V)$ for all $n\in \Z$.
\end{theorem}
\goodbreak

% or local henselian. 
Our second main theorem  %The following result 
generalizes a result of Balmer and Gille 
\cite{BalmerGille}, from regular to singular schemes, because
$K_{-1}(X)=0$ when $X$ is regular and separated.
However, their proof uses Localization for 
Witt groups and Devissage, both of which fail for singular schemes.

Let $\Tate(K_{-1}X)$ denote the $\Z_2$-Tate spectrum of the 
abelian group $K_{-1}(X)$ with respect to the standard involution 
(sending a vector bundle to its dual).
Its homotopy groups are the Tate cohomology groups 
$\tate^*(\Z_2,K_{-1}X)$.
%$\widehat{H}^*(\Z_2,K_{-1}(X))$. 

\begin{theorem}\label{thm:LCn}
Let $X$ be a scheme with an ample family of line bundles, such that 
$2\in\cO(X)^\times.$   %$2 \in \Gamma(X,O_X)^*$.
Then: \\
(i) There is a homotopy fibration of spectra for each $n\ge1$:
\[
S^{-n-1} \wedge \Tate(K_{-1}X) \longrightarrow L(X) \oplus L^{[1-n]}(X)
\longrightarrow L(X\times(\A^n\!-\!0)).
\]
(ii) Suppose that $K_{-1}X=0$, or more generally that
$\tate^i(\Z_2,K_{-1}X)=0$ for $i=0,1$. Then
%the groups $W^r(X\times(\A^n\!-\!0))$ are 4-periodic in $n$, and
\[
W^r(X\times(\A^n\!-\!0)) \cong W^r(X) \oplus W^{r+1-n}(X).
\]
When $n=1$ this becomes
$W^r(X\times \Gm)\cong W^r(X)\oplus W^{r}(X)$,
and specializes for $r=0$ to: $W(X\times \Gm)\cong W(X)\oplus W(X)$.
\end{theorem}

Note that taking homotopy groups of part (i) yields part (ii).  The main
ingredient in proving Theorem \ref{thm:LCn} is Theorem \ref{thm:GWCn} in the
text which gives a direct sum decomposition of the hermitian $K$-theory of $X
\times (\A^n -0)$ into four canonical pieces generalising Bass' Fundamental
Theorem for Hermitian $K$-theory \cite[Theorem 9.13]{Schlichting.Fund}.

We also prove parallel results for the higher Witt groups $W_i(X)$
(and coWitt groups) defined by the second author (see \cite{MKAnnalsH}),
and their variants  $W_i^{[r]}(X)$.
These results include homotopy invariance and:

\begin{proposition}\label{intro:higherW}
Let $X$ be a scheme with an ample family of line bundles, such that 
$2\in\cO(X)^\times.$
Then the higher Witt groups satisfy:
%if $n\equiv1\pmod4$ they equal
%$W^{[r]}_i(X)\oplus W^{[r-1]}_{i-1}(X)$.
%In general,
\[
W^{[r]}_i(X\times (\A^n\!-\!0)) \cong 
W^{[r]}_i(X) \oplus W^{[r-n]}_{i-1}(X).
\]
In particular, the higher Witt groups $W^{[r]}_i(X\times(\A^n\!-\!0))$ 
are 4-periodic in $n$.
%the classical Witt group $W(X \times (\A^n\!-\!0))$ is 4-periodic in n.
\end{proposition}

Since $W^{[0]}_0 = W$ is the classical Witt group and $W^{[r]}_i$ 
is 4-periodic in $r$ we obtain the following.

\begin{corollary}
For $X$ as in Proposition \ref{intro:higherW},
the classical Witt group $W(X \times (\A^n\!-\!0))$ 
is 4-periodic in $n$.
\end{corollary}

\vspace{1ex}
Here is a short overview of the contents of this paper.  
In Section \ref{sec:WittHtpyInv} we establish homotopy invariance of Witt and
stabilized Witt groups and prove Theorem \ref{thm:hi}.  
In Section \ref{sec:BassFundThm} we give an elementary proof of Theorem
\ref{thm:LCn} (ii) when $n=1$ based on Bass' Fundamental Theorem for
Grothendieck--Witt groups.  In Section \ref{sec:K(Gm)} we compute the
$K$-theory of $X \times (\A^n-0)$.  
In Section \ref{sec:L(Gm)} we compute the Witt and Grothendieck--Witt groups of
$X \times (\A^n-0)$ and prove Theorem \ref{thm:LCn}.  
In Section \ref{sec:node} we compute the Witt groups of a nodal curve over a
field of characteristic not $2$ and show that the formula of Balmer--Gille
\cite{BalmerGille} does not hold for this curve.  
In Section \ref{sec:HigherW} we generalise Theorems \ref{thm:hi} and
\ref{thm:LCn} to the higher Witt and coWitt groups of the  first author,
see Proposition \ref{intro:higherW}.
Finally, in the Appendix the second author computes the higher
Grothendieck--Witt groups of $X \times \P^n$ in a form that is needed in the
proofs in Section \ref{sec:L(Gm)}.  This generalises some unpublished results
of Walter \cite{WalterProj}.
\vspace{2ex}

\noindent
{\bf Notation.}
Following \cite{BalmerGille}, we write $C_n = \A^n-0$ 
for the affine $n$-space minus the origin.

%If $F$ is a functor from schemes to spectra, and 
%$Z$ is a closed subscheme of $X$, we will write 
%$F(X,Z)$ and $F(X\on Z)$ for the homotopy fibers of $F(X) \to F(Z)$ and 
%$F(X) \to F(X-Z)$, respectively.

%
Here are the various spectra associated to a scheme $X$ that we use.

For any abelian group $A$ with involution (or more generally a
spectrum with involution), we write $\Tate(A)$ for the (Tate) spectrum
representing Tate cohomology of the cyclic group $\Z_2$ with
coefficients in $A$.  If $A$ is a spectrum, then $\Tate(A)$ is the
homotopy cofiber of the hypernorm map $A_{hG}\to A^{hG}$; 
if $A$ is a group then $\pi_i\Tate(A) = \tate^i(\Z_2,A)$.

We write $\K(X)$ for the nonconnective $K$-theory spectrum; the groups
$K_i(X)$ are the homotopy groups $\pi_i\K(X)$ for all $i\in\Z$. 
In particular, $K_{-1}(X)=\pi_{-1}\K(X)$.
(See  \cite{WK, TT} for example.) We shall write $K^Q(X)$ for
Quillen's connective $K$-theory spectrum, and $K_{<0}(X)$
for the cofiber of the natural map $K^Q(X)\to\K(X)$.

There is a standard involution on these $K$-theory spectra,
and their homotopy groups $K_i(X)$, induced by the functor on 
locally free sheaves sending $\cE$ to its dual sheaf 
$\cE^* = \Hom_{\cO_X}(\cE,\cO_X)$; the corresponding Tate cohomology 
groups $\tate^i(\Z_2,K_jX)$ are written as $k_j$ and $k'_j$ in 
the classical ``clock'' sequence \cite[p.\,278]{MKAnnalsH}.
There are other involutions on $K$-theory; we shall write 
$K_i^{[r]}(X)$ for $K_i(X)$ endowed with the involution obtained using duality
with values in $\cO_X[r]$; see \cite[1.12]{Schlichting.Fund}. 

The second author defined the {\it Grothendieck--Witt} spectra $GW^{[r]}(X)$
and the {\it Karoubi--Grothendieck--Witt} spectra $\GW^{[r]}(X)$;
see \cite[5.7 and 8.6]{Schlichting.Fund}; the element
$\eta\in GW^{[-1]}_{-1}(\Z[1/2])$ plays an important role.

$L^{[r]}(X)$ denotes the spectrum obtained from $GW^{[r]}(X)$
by inverting $\eta$; see \cite[Def.\,7.1]{Schlichting.Fund}.
The homotopy groups $\pi_iL^{[r]}(X)$ are Balmer's 4-periodic
triangular Witt groups $L_i^{[r]}(X) = W^{r-i}(X)$
\cite[7.2]{Schlichting.Fund}.  In particular, $W^0(X)$ is the
classical Witt group of Knebusch \cite{Knebusch} of symmetric bilinear
forms on X and $W^2(X)$ is the classical Witt group of symplectic
(that is, $-1$-symmetric) forms on $X$.  The groups $W^1(X)$ and
$W^3(X)$ are the Witt groups of formations $(M,L_1,L_2)$ on $X$,
where $L_1$ and $L_2$ are Lagrangians on the form $M$; see
\cite[p.147]{Ranicki} or \cite{Walter}. 
When $X=\Spec R$ is affine, then these groups are also Ranicki's
$L$-groups $L_i(R) = W^{-i}(R)$; see \cite{Ranicki}.

The stabilized $L$-theory spectrum $\L^{[r]}(X)$ is the spectrum
obtained from $\GW^{[r]}(X)$ by inverting $\eta$; see
\cite[8.12]{Schlichting.Fund}.  It is better behaved than $L^{[r]}(X)$, 
as $\L^{[r]}(X)$ satisfies excision, as well as Zariski descent 
(for open subschemes $U$ and $V$ in $X=U\cup V$); see
\cite[9.6]{Schlichting.Fund}. 

%By definition \cite[B.12]{Schlichting.Fund}, 
%the Tate spectrum $\Tate(G,\mathbb E)$ 
%of a $G$--spectrum $\mathbb E$ is the homotopy cofiber of the
%canonical map $\mathbb E_{hG}\to \mathbb E^{hG}$.

These spectra fit into the following morphism of fibration sequences.
%along with the homotopy orbit spaces of the
%Quillen $K$-theory spectrum $K^Q(X)$ and the
%non-connective $K$-theory spectrum $\K(X)$.
(See \cite[7.6 and 8.13]{Schlichting.Fund}.)
\begin{equation}\label{K-GW-L}
\xymatrix@R=1.5em{
K^Q(X)^{[r]}_{h\Z_2} \ar[r]\ar[d]& GW^{[r]}(X) \ar[r]\ar[d]& L^{[r]}(X)\ar[d] \\
\K(X)^{[r]}_{h\Z_2} \ar[r] & \GW^{[r]}(X) \ar[r]& \L^{[r]}(X).
}\end{equation}

\begin{lemma}\label{H-L-L} 
There is a fibration sequence
\begin{equation*}
S^{-1}\!\wedge\Tate(K_{<0}^{[r]}(X)) \to
L^{[r]}(X) \to \L^{[r]}(X) \to \Tate(K_{<0}^{[r]}(X)).
\end{equation*}
\end{lemma}

\begin{proof}
This follows from \eqref{K-GW-L}, since
the cofiber of the left vertical map is $\K_{<0}^{[r]}(X)_{h\Z_2}$,
and the cofiber of the middle vertical map %in \eqref{K-GW-L}
is the homotopy fixed point spectrum $K_{<0}^{[r]}(X)^{h\Z_2}$;
see \cite[8.14]{Schlichting.Fund}.
\end{proof}

\section{Homotopy invariance of Witt groups}
\label{sec:WittHtpyInv}

In order to prove homotopy invariance of $L(X)$ 
(Theorem \ref{thm:hi}), we first establish homotopy invariance of $\L(X)$.

%In the (quasi-projective) scheme case, we need to consider
%several spectra associated to $X$. 
%Schlichting defined spectra $L^{[r]}(X)$ for all $n\in\Z$ in 
%\cite[7.1]{Schlichting.Fund}
%\edit{Ranicki's $L^{[r]}$?}
%such that the homotopy groups $L^{[r]}_i(X)$ are just Balmer's Witt
%groups $W^{n-i}(X)$ \cite[7.2]{Schlichting.Fund}. He also defined the
%stabilized spectrum $\L^{[r]}(X)$ in \cite[8.12]{Schlichting.Fund}.

\begin{lemma}\label{L(X[t])}
Let $X$ be a scheme with 
%noetherian quasi-projective (or more generally have 
an ample family of line bundles over $\Z[1/2]$. Then
for every vector bundle $V$ over $X$, the projection $V \to X$ induces
an equivalence of spectra
\[
\L^{[r]}(X) \map{\simeq} \L^{[r]}(V).  %(X\times\A^1).
\]
\end{lemma}

\begin{proof}
Since $\L^{[r]}$ satisfies Zariski descent, 
we may assume that $X=\Spec(A)$ is affine, that $V$ is trivial, 
and even that $V=\Spec(A[t])$.  %X\times\A^1$.
In this case, the homotopy groups 
$\L^{[r]}_i(A)$ and $\L^{[r]}_i(A[t])$
are the colimits of ordinary Witt groups 
$\L^{[r]}_i(A) = \colim W^{j-i}(A_j)$ and 
$\L^{[r]}_i(A[t]) = \colim W^{j-i}(A_j[t])$,
by \cite[7.2, 8.12]{Schlichting.Fund}.
These colimits are isomorphic because
$W^{*}(A)\cong W^{*}(A[t])$ %for each $j$ 
by \cite[3.10]{MKlocalisation}.
Hence $\L^{[r]}_i(A)\cong\L^{[r]}_i(A[t])$.
\end{proof}

\begin{proof}[Proof of Theorem \ref{thm:hi}]
Write $L(V,X)$ for the cofiber of $L(X)\to L(V)$, 
and similarly for the cofibers of $\L$ and $K_{<0}$.
By Lemma \ref{H-L-L}, we have a fiber sequence
\[
L^{[r]}(V,X) \to \L^{[r]}(V,X) \to \Tate(K_{<0}^{[r]}(V,X)).  
\]
The middle term is zero by Lemma \ref{L(X[t])}.
%homotopy invariance of $\mathbb{L}$.
The right hand term is zero because $K_{<0}(V,X)$ 
has uniquely $2$-divisible homotopy groups \cite{W-MV};
see  \cite[B.14]{Schlichting.Fund}.
This implies that $L^{[r]}(V,X)=0$, i.e., $L^{[r]}$ is homotopy invariant.
\end{proof}

\section{Bass' Fundamental Theorem for Witt groups}
\label{sec:BassFundThm}

The map $GW^{[r]}_i(X)\to\GW^{[r]}_i(X)$ is an isomorphism for 
all $i\ge0$ and all $n$; 
see \cite[9.3]{Schlichting.Fund}. For $i=-1$, we have the following result.
%($\GW(X)$ is called the {\it Karoubi--Grothendieck--Witt} spectrum.)

\begin{lemma}\label{8.14}
If $K_{-1}(X)\!=0$, then 
$ W_0(X)\!\cong \!GW^{[-1]}_{-1}(X)\!\smap{\cong}\!\GW^{[-1]}_{-1}(X)$.
\end{lemma}

\begin{proof}
%Let $K^Q_n(X)$ denote the connective Quillen $K$-theory of $X$,
%and $K_n(X)$ the (Bass) non-connective spectrum.
Since $K^Q_i(X)\cong\! K_i(X)$ for $i\ge0$, we see from
\cite[B.9]{Schlichting.Fund} that $\pi_0(K_{<0}^{h\Z_2}X)=0$ and
$\pi_{-1}(K_{<0}^{h\Z_2}X)\cong\! H^0(\Z_2,K_{-1}(X))$. Hence the middle
column of \eqref{K-GW-L} %\cite[8.14]{Schlichting.Fund} 
yields an exact sequence
\[
0 \to GW^{[-1]}_{-1}(X) \to \GW^{[-1]}_{-1}(X) \to H^0(\Z_2,K_{-1}(X)) = 0.
\]
Finally, $W^0(X)\cong GW^{[-1]}_{-1}(X)$ by \cite[6.3]{Schlichting.Fund}.
\end{proof}

\begin{theorem}\label{Fund.Thm-schemes}
Let $X$ be a quasi-projective scheme over $\Z[1/2]$. If 
$K_{-1}(X)=0$,  %$K_0(X)=K_0(X\times\Gm)$, 
or more generally $H^*(\Z_2,K_{-1}(X))=0$, then
%$k_{-1}=k'_{-1}=0$, then 
\[ W_0(X)\oplus W_0(X) \cong W_0(X\times\Gm). \]
\end{theorem}

%In the scheme case, 
%We have a map of spectra 
%$GW^{[r]}(X)\to\GW^{[r]}(X)$, with $GW^{[r]}_i(X)\to\GW^{[r]}_i(X)$
%isomorphic for $i\ge0$; %and injective for $i=0$; 
%see \cite[9.3]{Schlichting.Fund}.
%($\GW(X)$ is called the {\it Karoubi--Grothendieck--Witt} spectrum.)

\begin{proof}
The second author proved in \cite[9.13--14]{Schlichting.Fund} that 
there is a natural split exact ``contracted functor'' sequence 
(for all $n$ and $i$)
\begin{equation*}\begin{split}
0 \to \GW^{[r]}_i(X) &\to \GW^{[r]}_i(X[t])\oplus \GW^{[r]}_i(X[1/t]) \\ 
          &\to \GW^{[r]}_i(X[t,1/t]) \to \GW^{[r-1]}_{i-1}(X) \to 0.
\end{split}\end{equation*}
Taking $n=i=0$ (so $GW_0\cong\GW_0$), we get a natural split exact sequence
\begin{equation*}\begin{split}
0 \to GW_0(X) &\to GW_0(X[t])\oplus GW_0(X[1/t]) \\
            & \to GW_0(X[t,1/t]) \to \GW^{[-1]}_{-1}(X) \to 0.
\end{split}\end{equation*}
When $K_{-1}(X)=0$ we also have a split exact sequence
\cite[X.8.3]{WK}:
\[
0\to K_0(X) \to K_0(X[t])\oplus K_0(X[\frac1t]) \to K_0(X[t,\frac1t])\to 0.
\]
Mapping this to the $GW$-sequence, 
we have a split exact sequence on cokernels:
\[
0\!\to\! W_0(X)\!\to W_0(X[t])\oplus W_0(X[\frac1t])\!\to W_0(X[t,\frac1t]) \to
\GW^{[-1]}_{-1}(X) \to\!0.
\]
By Lemma \ref{8.14}, %and \cite[6.3]{Schlichting.Fund}, 
we have
\[  GW^{[-1]}_{-1}(X)\cong\GW^{[-1]}_{-1}(X)\cong W^0(X). \] 
%by \cite[6.3]{Schlichting.Fund}.
Since $W_*(X)\cong W_*(X[t])$ by Theorem \ref{thm:hi}, %W-homotopy,
the result follows.
\end{proof}

\section{$K$-theory of punctured affine space}
\label{sec:K(Gm)}

The following result generalizes the ``Fundamental Theorem'' of
$K$-theory, which says that there is 
an equivalence of spectra %a fibration sequence %is the case $n=1$.
\[
\K(X)~ \oplus~ \susp\K(X) ~\oplus N\K(X)\oplus N\K(X) 
~\map{\simeq}\ \K(X\times\Gm).
\]
Write $V(1)$ for the  vector bundle $\mathbf V(\cO(1))$
on $\P^{n-1}_X$ associated to the invertible sheaf $\cO(1)$.
For simplicity,
we write $\K(\A^n_X,X)$ for the fiber of $\K(\A^n_X)\to\K(X)$ induced
by the inclusion of $X$ as the zero-section of $\A^n_X$, and write
$\K(V(1),{\P^{n-1}_X})$ for the fiber of the map $\K(V(1))\to \K(\P^{n-1}_X)$
which is induced by the inclusion of $\P^{n-1}_X$ as the zero-section of $V(1)$.

\begin{theorem}
%\label{prop:KCn}
\label{thm:KCn}
For every integer $n\ge1$ and every quasi-compact and quasi-separated
scheme $X$, there is an equivalence of spectra 
\[
\K(X)~ \oplus~ \susp\K(X) ~\oplus~  \K(V(1),{\P^{n-1}_X}) ~ \oplus ~ \K(\A^n_X,X)~
\stackrel{\simeq}{\longrightarrow} \K(X\times C_n)
\] 
functorial in $X$.
%where $\eps$ is a map of
%spectra $\eps: \Omega \K(V(1),{\P^{n-1}_X}) \to \K(\A^n_X,X)$.
%and $\P^{n-1}_X$ (resp., $X$) is the image of the zero sections of the vector
%bundle $V(1)$ (resp., $\A^n_X$).
\end{theorem}

\begin{subremark}
If $X$ is regular, Theorem \ref{thm:KCn} is 
immediate from the fibration sequence
$\K(X) \to \K(X\times\A^n) \to \K(X\times C_n)$; 
see \cite[V.6]{WK}.
\end{subremark}

If $Z$ is a closed subscheme of a scheme $X$, we write
$\K(X\on Z)$ for the homotopy fiber of $\K(X) \to \K(X-Z)$.

\begin{proof}
Consider the points $0=(0,...,0)$ of $\A^n$ and
$z = [1:0:\cdots:0]$ of $\P^n$, and consider $\A^n$ embedded into $\P^n$ via
the open immersion $(t_1,...,t_n) \mapsto [1:t_1,\cdots,t_n]$ sending $0$ to $z$.
We will write $0_X$  (resp., $\zx$) for the corresponding subschemes 
of $\A_X^n$ (resp., $\P^n_X$); both are isomorphic to $X$. 
We have a commutative diagram of spectra
\addtocounter{equation}{-1}
\begin{subequations}
\renewcommand{\theequation}{\theparentequation.\arabic{equation}}
\begin{equation}\label{eq:C_n}
\xymatrix{
\K(\P^n_X\on \zx) \ar[d]_{\wr} \ar[r] & \K(\P^n_X) \ar[d] & \\ 
\K(\A^n_X\on 0_X) \ar[r]^{\hspace{3ex}\eps} & \K(\A^n_X) \ar[r] &  \K(X\times C_n)
}
\end{equation}
in which the lower row is a homotopy fibration, by definition, and the left vertical arrow is an equivalence, by Zariski-excision \cite{TT}.
%\end{subequations}
% \to \K(\A^n_X\on 0 \times X)[1].$$
%Consider the embedding $\A^n=\P^n-1_{1_{P^n}}\subset \P^n$ where $1_{\P^n}=$.
%The proposition will follow from Lemma \ref{lem:Triangle} once 
%we exhibit the required direct sum decompositions of the first two terms
%in \eqref{eq:C_n}.
We will first show that $\eps=0$, that is, we will exhibit a null-homotopy of
$\eps$ functorial in $X$.

Recall that $\K(\P^n_X)$ is a free $\K(X)$-module of rank $n+1$ on the basis
\begin{equation}\label{eq:b_n}
b_r = \bigotimes_{i=1}^r\left(\cO_{\P^n}(-1) \stackrel{T_i}{\longrightarrow}
\cO_{\P^n}\right),\hspace{2ex} r=0,...,n
\end{equation}
where $\cO_{\P^n}$ is placed in degree $0$ and $\P^n=\Proj(\Z[T_0,...,T_n])$.
By convention, the empty tensor product $b_0$ is the tensor unit $\cO_{\P^n}$.
Note that the restriction of $b_r$  to $\A^n$ is trivial in $K_0(\A^n)$ for $r=1,...,n$.
This defines a null-homotopy of the middle vertical arrow of (\ref{eq:C_n}) on the components $\K(X)\cdot b_r$ of $\K(\P^n_X)$ for $r=1,...,n$. 
The remaining component $\K(X) \cdot b_0$ maps split injectively into $\K(X \times C_n)$ with retraction given by any rational point of $C_n$.
Since the composition of the two lower horizontal maps is naturally null-homotopic, this implies $\eps =0$.
Thus, we obtain a functorial direct sum decomposition
$$\K(X \times C_n) \simeq \K(\A^n_X) ~\oplus ~\susp \K(\P^n_X\on \zx)$$
and it remains to exhibit %we are left exhibiting 
the required decomposition of the two summands.

The composition $0 \to \A^n \to \pt$ is an isomorphism
and induces the direct sum decomposition of $\K(\A^n_X)$ as
$\K(X) \oplus \K(\A^n_X,X)$. %of $\K(\A^n_X)$.  
For the other summand, note that $b_n$ has support in $z=V(T_1,...,T_n)$, %$t_0$.
as it is the Koszul complex for $(T_{1},...,T_n)$.
In particular, the composition
\[\xymatrix{ 
\K(X) \ar[r]^{\hspace{-5ex}\otimes b_n} & 
%\K(\P^n_X \on X\times t_0) 
\K(\P^n_X \on \zx) \ar[r] & \K(\P^n_X)
}\]
is split injective, and defines a direct sum decomposition
\[
% \K(\P^n_X \on X\times t_0)
\K(\P^n_X \on \zx) \cong \K(X) \oplus 
\widetilde{\K}(\P^n_X \on \zx).
\]
It remains to identify $\widetilde{\K}(\P^n_X \on \zx)$ with $\Omega\K(V(1),\P^{n-1}_X)$.
%A linear translation gives an equivalence  $\K(\A^n_X,1_X) \cong \K(\A^n_X,0_X)$.  
Consider the closed embedding
$j:\P^{n-1}=\Proj(\Z[T_1,...,T_{n}]) \subset \P^n$.  
Since $j(\P^{n-1})$ lies in $\P^n-\{z\}$,
we have a commutative diagram of spectra
\begin{equation}\label{eqn:bnj*}
\xymatrix{ 
\K(X) \ar[r]^1 \ar[d]_{b_n} & \K(X) \ar[d]^{b_n} \ar[r]& 0 \ar[d]\\ 
\K(\P^n_X \on \zx) \ar[r] \ar[d] & \K(\P^n_X) \ar[r]
  \ar[d]^{j^*} & \K(\P^n_X - \zx)\ar[d]^{j^*} \\ 
%\K(\P^n_X - X\!\times\!t_0)\ar[d]^{j^*}
0 \ar[r] &  \K(\P^{n-1}_X) \ar[r]_1 & \K(\P^{n-1}_X) 
}\end{equation} 
\end{subequations}
in which the rows are homotopy fibrations, and the middle column is split
exact.  It follows that we have a fibration
\[
%\K(\P^n_X \on X\!\times\!t_0) \to \K(\P^{n-1}_X) \to \K(\P^n_X - X\!\times\!t_0).
\widetilde{\K}(\P^n_X \on \zx) \longrightarrow  \K(\P^n_X - \zx) \stackrel{j^*}{\longrightarrow} \K(\P^{n-1}_X) .
\]
Since $V(1)\to\P^{n-1}_X$ is $\P^n_X-\zx\to \P^{n-1}_X$,
it follows that
\[
\susp \widetilde{\K} (\P^n_X \on \zx) \simeq 
\K(\P^n_X - \zx,\P^{n-1}_X) =  \K(V(1), \P^{n-1}_X).
\qedhere
\]
\end{proof}

\begin{remark}
The  proof of Theorem \ref{thm:KCn} also applies to
the homotopy $K$-theory spectrum $KH$ of \cite{WK}.
Since $KH(\A^n_X,X)=0$ and $KH(V(1)_X,\P^{n-1}_X) =0$, 
by homotopy invariance, we obtain an equivalence:
%Then the proof of Proposition \ref{prop:KCn} yields an equivalence
\[
KH(X) \oplus \susp KH(X) %S^1\wedge KH(X) 
\simeq KH(X\times C_n).
\]
%since $KH(\A^n_X,X) = KH(V(1)_X,\P^{n-1}_X) =0$, by homotopy invariance.
\end{remark}

The following fact will be needed in the next section.

\begin{lemma}
\label{lem:2DivCone}
Let $V$ be a vector bundle over a scheme $X$ defined over $\Z[1/2]$.
Then the homotopy groups of $\K(V,X)$ are uniquely $2$-divisible.
In particular, for any involution on $\K(V,X)$ we have $\Tate(\K(V,X))=0$.
\end{lemma}

\begin{proof}
This follows from \cite{W-MV} and Zariski-Mayer-Vietoris \cite{TT}.
\end{proof}

\section{$GW$ and $L$-theory of punctured affine space}
\label{sec:L(Gm)}

A modification of the argument in Theorem \ref{thm:KCn} yields a
computation of the Hermitian $K$-theory of $X\times C_n$.
This generalises the case $n=1$ of \cite[Theorem 9.13]{Schlichting.Fund}.

For any line bundle $\scL$ on $X$, we write $\K^{[r]}(X;\scL)$ 
(resp., $\GW^{[r]}(X;\scL)$) for the $K$-theory spectrum of $X$ 
(resp., $\GW$-spectrum) with involution $E \mapsto \Hom(E,\scL[r])$.
If $p:V \to X$ is a vector bundle on $X$, then $p$ embeds $\GW^{[r]}(X;\scL)$
into $\GW^{[r]}(V;\scL)$ 
as a direct summand with retract given by the zero
%\edit{typo fixed} 
section.  We write $\GW^{[r]}(V,X;\scL)$ for the complement of
$\GW^{[r]}(X;\scL)$ in $\GW^{[r]}(V;p^*\scL)$.  If $\scL=\cO_X$, we simply
write $\K^{[r]}(X)$, $\GW^{[r]}(X)$ and $\GW^{[r]}(V,X)$.

\begin{theorem}
%\label{prop:GWCn}
\label{thm:GWCn}
For all integers $r,n$ with $n\geq 1$ and every scheme $X$ over $\Z[1/2]$ 
with an ample family of line bundles, 
there is a functorial equivalence of spectra
\[  %S^1\wedge
\renewcommand\arraystretch{2}
\begin{array}{cl}
\GW^{[r]}(X\times C_n) \simeq 
&\phantom{\oplus~} \GW^{[r]}(X) \oplus~ \susp\GW^{[r-n]}(X)\\
 & \oplus~ \GW^{[r]}(V(1),{\P^{n-1}_X}; \cO(1\!-\!n)) 
 ~\oplus~  \GW^{[r]}(\A^n_X,X).
 \end{array}
\]
\end{theorem}

\begin{proof}
The proof is the same as that of Theorem \ref{thm:KCn} with the following
modifications. 
For space reasons, we write $\scL$ (resp., $\scL'$)
for the sheaf $\cO(1-n)$ on $\P^n$ (resp., on $\P^{n-1}$).
Consider the commutative diagram of spectra,
analogous to \eqref{eq:C_n},
\addtocounter{equation}{-1}
\begin{subequations}
\renewcommand{\theequation}{\theparentequation.\arabic{equation}}
\begin{equation}
\label{eq:GW(C_n)}
\xymatrix{
\GW^{[r]}(\P^n_X\on \zx; \scL) \ar[d]_{\wr} \ar[r] & \GW^{[r]}(\P^n_X; \scL) \ar[d] & \\ 
\GW^{[r]}(\A^n_X\on 0_X) \ar[r]^{\hspace{3ex}\eps} & \GW^{[r]}(\A^n_X) \ar[r] &  \GW^{[r]}(X\times C_n)
}
\end{equation}
in which the lower row is a homotopy fibration (by definition), and the left
vertical arrow is an equivalence, by Zariski-excision \cite[Thm.\,3]
{SchlichtingMV}, noting that $\scL=\cO(1-n)$ is trivial on $\A^n_X$.
Again, we will first show that $\eps = 0$.

Recall the complexes $b_i$ from (\ref{eq:b_n}).
We equip $b_0=\cO_{\P^n}$ with the unit form 
$\cO_{\P^n} \otimes \cO_{\P^n} \to \cO_{\P^n}:x\otimes y \mapsto xy$.
The target of the quasi-isomorphism
\[\xymatrix{
b_n \cong b_n\otimes \cO \ar[r]^{1\otimes T_0}& b_n\otimes \cO(1)&
}
%b_n \cong b_n\otimes \cO
%\stackrel{1\otimes T_0}{\longrightarrow} b_n\otimes \cO(1)
\]
is canonically isomorphic to $b_n^*\otimes \scL[n]$, and 
endows $b_n$ with a non-degenerate symmetric bilinear form 
with values in $\scL[n]$.  %$\cO(1-n)[n]$.
In detail, the complex $\beta_i = (T_i:\cO(-1) \to \cO)$ with $\cO$ placed in degree $0$ is endowed with a symmetric form
with values in $\cO(-1)[1]$:
\[\xymatrix{
\beta_i \otimes \beta_i \ar[d]^{\ffi_i} & \cO(-2) \ar[r]^{\hspace{-5ex}\left( \begin{smallmatrix} T_i \\ -T_i\end{smallmatrix}\right)} \ar[d] & \cO(-1) \oplus \cO(-1) \ar[r]^{\hspace{7ex}(T_i,T_i)} \ar[d]^{(1,1)} & \cO \ar[d] \\
\cO(-1)[1] & 0 \ar[r] & \cO(-1) \ar[r] & 0.}
\]
Hence the tensor product $b_n = \bigotimes_{i=1}^n\beta_i$ is equipped with
a symmetric form with values in $\scL[n]$:
\begin{equation}
\label{eqn:bnexplicit}
\xymatrix{ b_n \otimes b_n \ar[rr]^{\bigotimes_{i=1}^n\ffi_i} && \cO(-n)[n]
  \ar[r]^{\hspace{-2ex}T_0} & \cO(1-n)[n]} =\scL[n].
\end{equation}
Of course, to make sense of the map $\otimes_{i=1}^n\ffi_i$ we have to
rearrange the tensor factors in $b_n \otimes b_n$ using the symmetry of the
tensor product of complexes given by the Koszul sign rule.
%We consider $b_n$ endowed with that form. 
Note that $b_n$ restricted to $\A^n$ is $0$ in $GW_0^{[n]}(\A^n)$ 
since it is the external product of the restrictions of $\beta_i$ to 
$\A^1 = \Spec \Z[T_i, 1/2]$ which are trivial in $GW^{[1]}_0(\A^1)$, 
by \cite[Lemma
  9.12]{Schlichting.Fund}.

By Corollary \ref{cor:GWProj}, the right vertical map of (\ref{eq:GW(C_n)}) is
$$\xymatrix{
\GW^{[r-n]} (X)  ~ \oplus ~\K(X)^{\oplus m} ~ \oplus ~ A  \ar[rrr]^{\hspace{10ex}(b_n,\  H(\oplus_{i=1}^m \cO(-i)) ,\ a)} &&& \GW^{[r]}(\A^n_X)}$$
where $m= \lfloor \frac{n-1}{2} \rfloor$ and $a:A \to \GW^{[r]}(\A^n_X)$ is either $b_0\circ H: \K(X) \to \GW^{[r]}(\A^n_X)$ or $b_0: \GW^{[r]}(X) \to \GW^{[r]}(\A^n_X)$ depending on the parity of $n$.
Since $\cO(i)$ is isomorphic to $b_0$ over $\A^n_X$, this map is equal to
$(b_n, b_0 \circ H,..., b_0\circ H, a)$.
In other words, the map $\eps$ factors through
$$\xymatrix{\GW^{[r-n]} (X)  ~ \oplus  \GW^{[r]}(X)^{\oplus m} ~ \oplus  \GW^{[r]}(X)\ar[rr]^{\hspace{14ex}(b_n,\  b_0^{m},\ b_0)} && \GW^{[r]}(\A^n_X).}$$
Changing basis and using the fact that $b_n=0 \in \GW_0^{[n]}(\A^n)$, this map is isomorphic to 
$$\xymatrix{\GW^{[r-n]} (X)  ~ \oplus  \GW^{[r]}(X)^{\oplus m} ~ \oplus  \GW^{[r]}(X)\ar[rr]^{\hspace{14ex}(0,\  0,\ b_0)} && \GW^{[r]}(\A^n_X).}$$
In other words, the map
$\eps$ factors through $b_0: \GW^{[r]}(X) \to \GW^{[r]}(\A^n_X)$.
Since the composition 
$$\GW^{[r]}(X) \to \GW^{[r]}(\A^n_X) \to \GW^{[r]}(X\times C_n)$$
 is split injective and the composition of the lower two horizontal arrows in (\ref{eq:GW(C_n)}) is zero, it follows that $\eps$ is null-homotopic functorially in $X$, and we obtain the functorial direct sum decomposition
 $$
\GW^{[r]}(X\times C_n) \simeq \GW^{[r]}(\A^n_X) ~ \oplus ~ \susp \GW^{[r]}(\P^n_X\on \zx; \scL)
$$

As before, the composition $0 \to \A^n \to \pt$ is an isomorphism
and induces the direct sum decomposition 
\[
\GW^{[r]}(\A^n_X) = \GW^{[r]}(X) \oplus \GW^{[r]}(\A^n_X,X).
\] 
To decompose the other direct summand we use the analogue of diagram \eqref{eqn:bnj*} which is:
\[
\xymatrix{
\GW^{[r-n]}(X) \ar[r]^1 \ar[d]_{\otimes b_n} & 
\GW^{[r-n]}(X) \ar[d]^{\otimes b_n} \ar[r]   &
  0 \ar[d]\\
\GW^{[r]}(\P^n_X\! \on\! \zx; \scL) \ar[r] \ar[d] & 
\GW^{[r]}(\P^n_X; \scL) \ar[r] \ar[d]^{j^*} & 
\GW^{[r]}(\P^n_X - \zx; \scL)\ar[d]^{j^*} \\
0 \ar[r] & \GW^{[r]}(\P^{n-1}_X; \scL') \ar[r]^1 & 
\GW^{[r]}(\P^{n-1}_X; \scL').
}\]
The rows are homotopy fibrations and the middle column is split exact, by
Theorem \ref{thm:Walter}.
It follows that $\GW^{[r-n]}(X)$ is a direct factor of 
$\GW^{[r]}(\P^n_X\!\on\!\zx;\!\scL)$
%$\GW^{[r]}(\P^n_X \on X\times t_0; \cO(1-n))$
with complement %$\widetilde{\GW}^{[r]}(\P^n_X \on X\times t_0; O(-n))$ 
equivalent to
\[
\Omega\GW^{[r]}(\P^n_X - \zx, \P^{n-1}_X; \scL) \cong
\Omega\GW^{[r]}(V(1), \P^{n-1}_X; \scL'),
\]
%\Omega\GW^{[r]}(\P^n_X - X\times t_0, \P^{n-1}_X; \cO(1-n)) \cong
%\Omega\GW^{[r]}(V(1), \P^{n-1}_X; \cO(1-n))
%\Omega\GW^{[r]}(\cO_{\P^{n-1}_X}(1), \P^{n-1}_X; O(1-n)) 
since the restrictions of $\scL$ and $\scL'$
%$\cO_{\P^n}(1-n)$ and $\cO_{\P^{n-1}}(1-n)$ 
along the embedding $\P^n\!-z \subset \P^n$ and the projection 
$(\P^n\!-z) \to \P^{n-1}$ are isomorphic for $n\geq 1$.
\end{subequations}
\end{proof}

When X is regular, the last two terms in Theorem \ref{thm:GWCn} vanish.
In this case, Theorem \ref{thm:GWCn} gives the exact computation of 
$\GW^{[r]}(X\times C_n)$ and hence of $W^{[r]}(X\times C_n)$; 
the latter recovers a result of Balmer--Gille, 
cf.\ \cite{BalmerGille}.

\begin{corollary}
For all integers $r,n$ with $n\geq 1$ and every noetherian regular 
separated scheme $X$ over $\Z[1/2]$, there is an equivalence of spectra
\[  %S^1\wedge
\GW^{[r]}(X) \oplus~ \susp\GW^{[r-n]}(X) 
\stackrel{\sim}{\longrightarrow} \GW^{[r]}(X\times C_n).
\]
\end{corollary}

\begin{proof}
Recall that a noetherian regular separated scheme has an ample
family of line bundles.  The corollary follows from Theorem
\ref{thm:GWCn} since $\GW$ is homotopy invariant on such schemes
\cite[Thm.\,9.8]{Schlichting.Fund}.
\end{proof}

Recall that the $L$-theory spectrum $L^{[r]}(X)$  and the 
stabilized $L$-theory spectrum $\L^{[r]}(X)$ are obtained from 
%the graded spectra
$GW^{[r+*]}(X)$ and $\GW^{[r+*]}(X)$ by inverting the element 
\[
\eta \in GW^{[-1]}_{-1}(\Z[1/2]) = \GW^{[-1]}_{-1}(\Z[1/2])= W(\Z[1/2])
\]
corresponding to $\langle 1\rangle \in W(\Z[1/2])$. See
\cite[Definitions 7.1 and 8.12]{Schlichting.Fund}.

\begin{remark}
\label{rmk:CupProdMap}
All maps in Theorem \ref{thm:GWCn} are $GW^{[*]}(\Z[1/2])$-module maps.
Therefore, the map on the second factor,
\[ %\tilde{b}_n \cup\ \ : 
S^1 \wedge \GW^{[r-n]}(X) \longrightarrow
\GW^{[r]}(X\times C_n)
\]
is the cup product
with an element 
$$\tilde{b}_n \in GW_1^{[n]}(\Spec(\Z[1/2])\times C_n) = \GW_1^{[n]}(\Spec(\Z[1/2])\times C_n).$$
Inverting $\eta$ therefore yields canonical maps
\[
(1,\tilde{b}_n): L^{[r]}(X) \oplus~ \susp L^{[r-n]}(X) \to L^{[r]}(X\times
C_n),
\]
\[
(1,\tilde{b}_n): \L^{[r]}(X) \oplus~ \susp \L^{[r-n]}(X) \to \L^{[r]}(X\times
C_n).
\]
\end{remark}

\begin{theorem}
\label{thm:stabLCn}
Let $X$ be a scheme over $\Z[1/2]$ with an ample family of line bundles.
Then the following map is an equivalence of spectra
$$(1,\tilde{b}_n): \L^{[r]}(X) ~\oplus~ \susp\L^{[r-n]}(X) %S^1\wedge
\stackrel{\sim}{\longrightarrow} \L^{[r]}(X\times C_n)$$
\end{theorem}

\begin{proof}
This follows from the $\GW$ formula in Theorem \ref{thm:GWCn} 
by inverting the element $\eta$ of $GW^{-1}_{-1}(\Z[1/2])$ and noting that
\[
\L^{[r]}(V(1),{\P^{n-1}_X}; \cO(1-n)) = \L^{[r]}(\A^n_X,X)=0
\]
by homotopy invariance of $\L$ (Lemma \ref{L(X[t])}).
\end{proof}
We can now deduce Theorem \ref{thm:LCn} from Theorem \ref{thm:stabLCn}.
% by means of Lemma \ref{H-L-L}.
Recall that $\K^{[r]}(X)$ denotes the spectrum $\K(X)$ endowed with the
involution obtained using duality with $\cO_X[r]$.

\begin{proof}[Proof of Theorem \ref{thm:LCn}]
The proof of Theorem \ref{thm:GWCn} shows that the equivalence in
Theorem \ref{thm:KCn} is $\Z_2$-equivariant.
In other words, the spectrum $\K^{[r]}(X\times C_n)$ with $\Z_2$-action is equivalent to
\[
\K^{[r]}(X) \oplus~ \susp \K^{[r-n]}(X) \oplus \K^{[r]}(V(1),{\P^{n-1}_X};\cO(1-n)) \oplus \K^{[r]}(\A^n_X,X).
\]
%where $\eps$ is a $\Z_2$-equivariant map of spectra,
%\[
%\eps: \Omega \K^{[r]}(V(1),{\P^{n-1}_X};\cO(1-n)) \to
%\K^{[r]}(\A^n_X,X).
%\]
We saw in Lemma  \ref{lem:2DivCone} that the Tate spectrum $\Tate$ of the last
two summands are trivial. 
% vanishes on $\Z_2$-spectra with uniquely $2$-divisible homotopy groups, 
Hence the map
\[
\K^{[r]}_{<0}(X) \oplus (S^1\wedge \K^{[r-n]}_{<0}(X)) \longrightarrow
\K^{[r]}_{<0}(X\times C_n)
\]
is an equivalence after applying $\Tate$. 
%in view of Lemma \ref{lem:2DivCone}.
Suppressing $X$, consider the  map of homotopy fibrations of spectra:
%In the map of homotopy fibrations of spectra (suppressing $X$)
\[
\xymatrix{
L^{[r]}\! \oplus \susp L^{[r-n]} \ar[r] \ar[d] &
%L^{[r]} \oplus S^1\wedge L^{[r-n]} \ar[r] \ar[d] &
\L^{[r]} \oplus \susp \L^{[r-n]} \ar[d]^{\simeq} \ar[r] & 
%\L^{[r]} \oplus S^1\wedge \L^{[r-n]} \ar[d]^{\simeq} \ar[r] & 
\Tate(\K^{[r]}_{<0}) \oplus  \susp \Tate(\K^{[r-n]}_{<0}) \ar[d] \\
%\Tate(\K^{[r]}_{<0}) \oplus  S^1\wedge \Tate(\K^{[r-n]}_{<0}) \ar[d] \\
L^{[r]}(C_n) \ar[r]&  \L^{[r]}(C_n) \ar[r]  & \Tate(\K^{[r]}_{<0}(C_n)).}
\]
The homotopy fiber of the right vertical map is 
\[
S^1\wedge \Tate(S^{-1}\wedge K_{-1}^{[r-n]}) = \Tate(K_{-1}^{[r-n]}) 
= S^{r-n}\wedge \Tate(K_{-1}).
\]
Hence the homotopy fiber of the left vertical map is 
$S^{r-n-1}\wedge \Tate(K_{-1})$.
The statement in the theorem is the case $r=0$.
\end{proof}

\section{The Witt groups of a node and its Tate circle}\label{sec:node}
%\label{sec:Example}

To give an explicit example where $W(R[t,1/t])\ne W(R)\oplus W(R)$, i.e.,
where the conclusion of Theorem \ref{thm:LCn}(ii) fails for $X=\Spec(R)$, we
consider the coordinate ring of a nodal curve over a field of characteristic
not $2$.

The following lemma applies to the coordinate ring $R$ of any curve, 
as it is well known that 
$K_{n}(R)=0$ for $n\le-2$. (See \cite[Ex.\,III.4.4]{WK} for example.)

\begin{lemma}\label{lem:node}
If $R$ is a $\Z[1/2]$-algebra with $K_i(R)=0$ for $i\le{-2}$, then
\[
L^{[0]}(R[t,1/t])\cong L^{[0]}(R)\oplus \L^{[0]}(R).
\]
\end{lemma}

\begin{proof}
The assumption that $K_i(R)=0$ for $i\le-2$ implies that 
\[
K_{<0}(R[t,1/t]) \cong K_{<0}(R) \oplus NK_{<0}(R) \oplus NK_{<0}(R).
\]
Now $\Tate(NK_{<0})=0$, because the homotopy groups of 
$NK_{<0}(R)$ are $2$-divisible (by \cite{W-MV}).
Hence $\Tate(K_{<0}R)\simeq \Tate(K_{<0}R[t,\frac{1}t])$
The lemma now follows from the following diagram, 
whose rows and columns are fibrations by Lemma \ref{H-L-L} and
Theorem \ref{thm:stabLCn},
and whose first two columns are split (by $t\mapsto1$).
\begin{equation*}
\xymatrix@R=1.0em{
%S^{-2}\!\wedge \Tate(K_{<0}R)\ar[r]\ar[d]^{\simeq}& 
L(R) \ar[r]\ar[d]^{\textrm{split}}  & \L(R)\ar[r]\ar[d]^{\textrm{split}}
&  \Tate(K_{<0}R) \ar[d]^{\simeq} \\
%S^{-2}\!\wedge \Tate(K_{<0}(R[t,\frac1t]))\ar[r]\ar[d]& 
L(R[t,\frac1t]) \ar[r]\ar[d]& \L(R[t,\frac1t])\ar[r]\ar[d]
& \Tate(K_{<0}(R[t,\frac1t])) \ar[d] \\
%  0\ar[r] &
 \textrm{cofiber} \ar[r]^{\simeq}\ar[r] & \L(R) \ar[r] & 0.
}%\qedhere
\vspace{-16pt}
\end{equation*}
\end{proof}

In what follows, $R$ will denote the node ring ($y^2=x^3-x^2$) over a field
$k$ of characteristic $\ne2$.  

If $F$ is any homotopy invariant functor from $k$-algebras to spectra
satisfying excision, the usual Mayer-Vietoris argument for $R\subset k[t]$
yields
$F(R) \simeq F(k)\oplus \Omega F(k)$; see \cite[III.4.3]{WK}.
In particular, 
\begin{equation}\label{eq:L(R)}
\L(R)\simeq \L(k)\oplus \Omega\,\L(k),
\quad  {\rm and} \quad KH(R) \simeq KH(k)\oplus \Omega\,KH(k).
\end{equation}
Since $KH_{<0}(k)=0$, $KH_0(k)=\Z$ we have
$K_{<0}(R) \simeq KH_{<0}(R) \simeq S^{-1}\wedge \Z.$
It follows that $\pi_i\Tate(K_{<0}(R)) = \tate^{i+1}(\Z)$.

\begin{remark}\label{W*(k)}
It is well known that $W^n(k)=0$ for $n\not\equiv0\pmod4$; the case
$W^2(k)=0$ (symplectic forms) is classical; a proof that
$W^1(k)=W^3(k)=0$ is given in \cite[Thm.\,5.6]{Balmer-II}, although
the result was probably known to Ranicki and Wall. 
Since $L(k)\simeq \L(k)$, $\L_n(R)=\L^{-n}(R)$ is:
$W(k)$ for $n\equiv0,3\pmod4$, and 0 otherwise.
\end{remark}

Recall that the fundamental ideal $I(k)$ is the kernel of the 
(surjective) rank map $W(k)\to\Z/2$.

\begin{lemma}\label{W*node}
When $R$ is the node, $W(R) \cong W(k) \oplus \Z/2$.\\
In addition, $W^{1}(R)\cong I(k)$ and $W^{2}(R)= W^{3}(R)=0$.
\end{lemma}

\begin{proof}
Since $\tate^0(\Z_2,\Z)=\Z/2$ and $\tate^1(\Z_2,\Z)=0$,  
Lemma \ref{H-L-L} and Remark \ref{W*(k)} yield the exact sequences:
\begin{align*} 
0\to L_3(R) \to \L_3(R) \to&~\tate^0(\Z_2,\Z) \to L_{2}(R) \to \L_{2}(R)\to0,\\
0\to L_1(R) \to \L_1(R) \to&~\tate^0(\Z_2,\Z) \to L_{0}(R) \to \L_{0}(R)\to0.
\end{align*}
Now the map $\L_3(R)\cong W^0(k)\to \tate^0(\Z_2,\Z)\cong\Z/2$ is the
%Now the map $W(k)\cong \L_3(R)\to \tate^0(\Z_2,\Z)\cong\Z/2$ is the
rank map $W(k)\to\Z/2$; it follows that $L_3(R)\cong I(k)$ and $L_2(R)=0$.
Since $\L_1(R)=0$, the second sequence immediately yields $W^3(R)=L_1(R)=0$.
Finally, the decomposition $W(R) \cong W(k) \oplus \Z/2$ follows
because the map $L_{0}(R) \to \L_{0}(R)\cong L_0(k)$ is a surjection,
split by the natural map $L_0(k)\to L_0(R)$.
\end{proof}

\begin{example}\label{ex:node}
%Let $R$ be the node ($y^2=x^3-x^2$) over a field $k$ of characteristic $\ne2$.
%It is well known that $K_{-2}(R)=0$ and 
%$K_{-1}(R)=\Z$ with the trivial involution. Hence Lemma \ref{lem:node} applies.

%Replacing $R$ by $R[t,1/t]$ yields
%similar identities. Hence $\L(R[t,1/t])\simeq\L(k)\times\susp\L(k)$.

By Lemma \ref{W*node}, $W(R)\cong W(k)\oplus \Z/2$.
Lemma \ref{lem:node} yields: 
\[
W^0(R[t,1/t])\cong W^0(R)\oplus \L^0(R) \cong W(k)\oplus \Z/2 \oplus W(k).
%W^0(R[t,1/t])\cong W^0(R)\oplus W^0(R) \oplus \Z/2.
\]
In addition, $W^1(R)\cong I(k)$ but $W^1(R[t,1/t])\cong I(k)\oplus W(k)$.
\end{example}

%\newpage
\section{Higher Witt groups}
\label{sec:HigherW}

Recall that the \emph{higher Witt group} $W_i(X)$ is defined to be
the cokernel of the hyperbolic map $\K_i(X)\to\GW_i(X)$;
see \cite{MKAnnalsH}.
More generally, we can consider the cokernel $W^{[r]}_i(X)$
of $\K_i(X)\to\GW^{[r]}_i(X)$. Similarly, one can define the 
\emph{higher coWitt group} ${W'}^{[r]}_i(X)$ to be 
the kernel of the forgetful map
$\GW^{[r]}_i(X)\to \K_i(X)$. In this section we show that  
$W^{[r]}_i(X)$ and ${W'}^{[r]}_i(X)$ are homotopy invariant and we compute their values on $X \times C_n$.

We begin with an observation.  As usual, for any functor $F$ from
schemes to spectra or groups, and any vector bundle $V\to X$,
we write $F(V,X)$ for the cofiber (or cokernel) of $F(X)\to F(V)$.
%let $NF(X)$ denote the cofiber (or cokernel)
%of $F(X)\to F(X\times\A^1)$.

%Because the homotopy groups $NK_i(X)$ are uniquely 2-divisible.
%Writing $NK^{[r]}_i(X)$ for these groups, endowed with the 
%involution arising from duality with $\cO[r]$, we have a
%natural decomposition of $NK^{[r]}_i(X)$ as the direct sum
%of its symmetric part $NK_i(X)^{[r]}_{+}=\pi_i(NK(X)_{h\Z_2})$
%and its skew-symmetric part $NK_i(X)^{[r]}_{-}$.

Recall that the homotopy groups $K_i(V,X)$ are uniquely 2-divisible.
Writing $K^{[r]}_i(V,X)$ for these groups, endowed with the 
involution arising from duality with $\cO[r]$, we have a
natural decomposition of $K^{[r]}_i(V,X)$ as the direct sum
of its symmetric part $K^{[r]}_i(V,X)_{+}=\pi_i(K(V,X)^{[r]}_{h\Z_2})$
and its skew-symmetric part $K^{[r]}_i(V,X)_{\,-\,}$\ .

\begin{lemma}\label{NK+}
For every vector bundle $V$ over $X$, and for all $i$ and $r$, 
$K^{[r]}_i(V,X)_{+} \cong \GW^{[r]}_i(V,X)$.
\end{lemma}

\begin{proof}
There is a fibration
$\K(V,X)^{[r]}_{h\Z_2} \to \GW^{[r]}(V,X) \to  \L^{[r]}(V,X)$;
see \eqref{K-GW-L} and \cite[8.13]{Schlichting.Fund}. 
%(where the homotopy groups of $\L^{[r]}(X)$ give the stablilized $L$-groups).
Since we proved in Lemma \ref{L(X[t])} %2.1
that $\L^{[r]}(V,X)=0$, we get an isomorphism of spectra 
$\K(V,X)^{[r]}_{h\Z_2} \map{\simeq}\ \GW^{[r]}(V,X)$ 
and hence group isomorphisms
$K^{[r]}_i(V,X)_{+} \cong \GW^{[r]}_i(V,X)$. 
\end{proof}

\begin{theorem}\label{thm:higher-hi}
Let $X$ be a scheme over $\Z[1/2]$ with an ample family of line bundles.
The higher Witt and coWitt groups are homotopy invariant 
in the sense that for every vector bundle $V$ over $X$,
the projection $V \to X$ induces
isomorphisms of higher Witt and coWitt groups:
\[
W^{[r]}_i(V)\cong W^{[r]}_i(X) \quad \mathrm{and} \quad
{W'_i}^{[r]}(V)\cong{W'_i}^{[r]}(X).
\]
\end{theorem}

\begin{proof}
The hyperbolic map $H$ is a surjection, as it factors:
\[
K_i(V,X)\ \map{\textrm{onto}}\ K^{[r]}_i(V,X)_{+}\ \map{\simeq}\ \GW^{[r]}_i(V,X).
\]
Hence the cokernel $W^{[r]}_i(V,X)$ is zero.
Similarly, the forgetful functor $\GW^{[r]}(V,X) \map{} \K(V,X)^{[r]}$
factors as the equivalence $\GW^{[r]}(V,X) \simeq \K(V,X)^{[r]}_{h\Z_2}$
followed by the canonical map $\K(V,X)^{[r]}_{h\Z_2} \to \K(V,X)^{[r]}$.
On homotopy groups, $\GW^{[r]}_i(V,X) \cong K^{[r]}_i(V,X)_{+} \to K^{[r]}_i(V,X)$
is an inclusion, so the kernel ${W'_i}^{[r]}(V,X)$ is zero.
\end{proof}

\medskip 
A similar argument applies to $W_*(X\times C_n)$.

\begin{theorem}
$W^{[r]}_i(X\times C_n) \cong W^{[r]}_i(X) \oplus W^{[r-n]}_{i-1}(X)$, and\\
${W'}^{[r]}_i(X\times C_n) \cong {W'}^{[r]}_i(X) \oplus {W'}^{[r-n]}_{i-1}(X)$.
\end{theorem}

\begin{proof}
As we saw in Sections \ref{sec:K(Gm)} and \ref{sec:L(Gm)},
the hyperbolic map 
$$H:\K(X\times C_n)\ \map{}\ \GW^{[r]}(X\times C_n)$$
is the sum of four maps. 
%It follows from Lemma \ref{NK+} 
%We saw in Theorem \ref{thm:stabLCn} (and its proof) 
%that the last two maps are split surjective. %an equivalence.
On homotopy groups, the cokernel $W^{[r]}_i(X\times C_n)$ of $H$
is the sum of the corresponding
cokernels.  The first two cokernels are $W^{[r]}_i(X)$ and
$W^{[r-n]}_{i-1}(X)$, while the last two are zero 
by Theorem \ref{thm:higher-hi}.

A similar argument applies to the coWitt groups, which are
the kernels of the map $F$.
\end{proof}

\newpage
\appendix
%\begin{section}
\section{Grothendieck--Witt groups of $\P^n_X$}
\centerline{Marco Schlichting}\smallskip

The goal of this appendix is to prove Theorem \ref{thm:Walter} which
was used in the proof of Theorem \ref{thm:GWCn}.  
As a byproduct we obtain a computation of the $\GW^{[r]}$-spectrum 
of the projective space $\P^n_X$ over $X$.
The $\pi_0$-versions are due to Walter \cite{WalterProj}, and
the methods of {\em loc.cit.}\ could be adapted to give a proof of
Theorem \ref{thm:Walter}.  Here we will give a more direct proof.
Using similar methods, a more general treatment of the Hermitian 
$K$-theory of projective bundles will appear in \cite{Rohrbach}.

Recall from \eqref{eq:b_n} and \eqref{eqn:bnexplicit} 
that there is a strictly perfect complex $b_n$ on $\P^n$ equipped with
a symmetric form $b_n \otimes b_n \to \sL[n]$, whose adjoint is 
a quasi-isomorphism; here $\sL$ is %and $\sL'$ are
the line bundle $\cO(1-n)$ on $\P^n$. %and $\P^{n-1}$. 
Let $j:\P^{n-1} \to \P^n$ denote the closed embedding as the vanishing
locus of $T_{0}$.

\begin{theorem}\label{thm:Walter}
  Let $X$ be a scheme over $\Z[1/2]$ with an ample family of line
  bundles, and let $n\geq 1$ be an integer.  Then the following
  sequence of Karoubi-Grothendieck--Witt spectra is a split fibration
  for all $r\in \Z$:
\[
\GW^{[r-n]}(X) \stackrel{\otimes b_n}{\longrightarrow} \GW^{[r]}(\P^n_X,\sL) 
\stackrel{j^*}{\longrightarrow} \GW^{[r]}(\P^{n-1}_X,j^*\sL). %\sL'.
\]
\end{theorem}

\noindent
The proof will use the following slight generalization of 
\cite[Prop.\,8.15]{Schlichting.Fund} (``Additivity for $\GW$''),
which was already used in the proof of the blow-up formula for $\GW$
in \cite[Thm.\,9.9]{Schlichting.Fund}.  If $(\cA,w,\vee)$ is
a dg category with weak equivalences and duality, %$(\cU,w,\vee)$ 
we write $\cT\cA$ for $w^{-1}\cA$, the associated triangulated category with
duality obtained from $\cA$ by formally inverting the weak equivalences; 
see \cite[\S1]{Schlichting.Fund}. The associated hyperbolic category 
is $\cH\cA=\cA\times\cA^{op}$, and $\GW(\cH\cA,w\times w^{op})\cong\K(\cA,w)$; 
see \cite[4.7]{Schlichting.Fund}.

\begin{lemma}
\label{lem:GWAddty}
Let $(\cU,w,\vee)$ be a pretriangulated dg category with weak
equivalences and duality such that $\frac{1}{2} \in \cU$. 
Let $\sA$ and $\cB$ be full pretriangulated dg subcategories of $\cU$
containing the $w$-acyclic objects of $\cU$. %$\cU^w$ of $\cU$.  
Assume that: (i) $\cB^{\vee}=\cB$; \\
(ii) $\cT\cU(X,Y)=0$ for all $(X,Y)$ in 
$\sA\times \cB$, $\cB \times \sA^{\vee}$ or
$\sA \times \sA^{\vee}$; and \\
(iii) $\cT\cU$ is generated as a triangulated
category by $\cT\sA$, $\cT\cB$ and $\cT\sA^{\vee}$.  
Then the exact dg form functor
\[
\cB \times \sH\sA \to \cU: X, (Y,Z) \mapsto X \oplus Y \oplus Z^{\vee}
\]
induces a stable equivalence of Karoubi--Grothendieck--Witt spectra: 
\[
\GW(\cB,w) \times \K(\sA,w) = \GW(\cB,w) \times \GW(\sH\sA,w\times w^{op}) \stackrel{\sim}{\longrightarrow} \GW(\cU,w).
\]
\end{lemma}

\begin{proof}
  Let $v$ be the class of maps in $\cU$ which are isomorphisms in
  $\cT\cU/\cT\cB$.  Then the sequence $(\cB,w) \to (\cU,w) \to (\cU,v)$
  induces a fibration of Grothendieck--Witt spectra 
\[
\GW(\cB,w) \to \GW(\cU,w) \to \GW(\cU,v)
\]
by the Localization Theorem \cite[Thm.\,8.10]{Schlichting.Fund}.  
Let $\cA'\subset \cU$ be the full dg
  subcategory whose objects lie in the triangulated subcategory of
  $\cT\cU$ generated by $\cT\cA$ and $\cT\cB$.  By Additivity for $\GW$
\cite[Prop.\,8.15]{Schlichting.Fund},
  the inclusion $(\cA',v) \subset (\cU,v)$ induces an equivalence of
  spectra $\K(\cA',v) \simeq \GW(\cU,v)$.  Finally, the map $(\cA,w) \to
  (\cA',v)$ induces an equivalence of associated triangulated
  categories and thus a $\K$-theory equivalence:
$\K(\cA,w) \map{\simeq}\K(\cA',v)\simeq\GW(\cU,v).$  
The result follows.
\end{proof}

For the proof of Theorem \ref{thm:Walter}, we shall need some notation.
Recall that $\sPerf(X)$ is the dg category of strictly perfect complexes on $X$,
and $w$ is the class of quasi-isomorphisms; the localization
$w^{-1}\sPerf(X)$ is the triangulated category $D^b\Vect(X)$.

Now consider the structure map $p\!:\P^m_X \to X$ for $m=n, n-1$,
and $\sL\!=\cO(1-n)$.
Recall that $D^b\Vect(\P^{m}_X)$ has a semi-orthongonal decomposition 
with pieces $\cO(i)\otimes p^*D^b\Vect(X)$, $i=0, -1, ...., -m$.
Let $\cU$ be the full dg subcategory of $\sPerf(\P^{n}_X)$
on the strictly perfect complexes on $\P^{n}_X$ lying in
the full triangulated subcategory of $D^b\Vect(\P^{n}_X)$
generated by $\cO(i)\otimes p^*D^b\Vect(X)$ for $i=0,-1,....,1-n$.
%here $p$ is the structure map $\P^n_X \to X$.
% and $\P^{n-1}_X \to  X$ in both cases.  
Note that $\cU$ is closed under the duality $\vee$ with values in
$\sL$.  
Finally, let $v$ denote the class of maps in 
$\sPerf(\P^{n}_X)$ which are isomorphisms in $D^b\Vect(\P^{n}_X)/\cT\cU$.
%Denote by $w$ the set of maps in $\sPerf(\P^{n}_X)$
%  and $\sPerf(\P^{n-1}_X)$ which are quasi-isomorphisms, and denote by
%  $v$ \Vect(\P^{n}_X)/\cT\cU$.

\begin{proof}[Proof of Theorem \ref{thm:Walter}]
Consider the following commutative diagram of 
Karoubi--Grothendieck--Witt spectra:
\[\xymatrix{
 & \GW^{[r-n]}(X) \ar[d]^{\otimes b_n} \ar[rd]^{\simeq} & \\
 \GW^{[r]}(\cU,w,\vee) \ar[dr]_{\simeq} \ar[r] &\GW^{[r]}(\P^n_X,\sL)
 \ar[d]^{j^*}\ar[r] & \GW^{[r]}(\sPerf \P^n_X, v,\vee) \\
  & \GW^{[r]}(\P^{n-1}_X,j^*\sL). & }
\]
The middle row is a homotopy fibration by Localization
\cite[Thm.\,8.10]{Schlichting.Fund}.  The
upper right diagonal arrow is a weak equivalence, because the standard
semi-orthogonal decomposition on $\P^n_X$ yields an equivalence of
triangulated categories $\otimes b_n: D^b\Vect(X) \map{\simeq}
D^b\Vect(\P^n_X)/\cT\cU$.  Finally, the lower left diagonal arrow is a
weak equivalence by Lemma \ref{lem:GWAddty}, where we choose the full
dg subcategories $\sA$, $\sA'$ and $\cB$, $\cB'$ of $\cU$ and
$\cU'=\sPerf(\P^{n-1}_X)$ as follows.  They are determined by their
associated triangulated categories.  

When $n=2m+1$ is odd, we let
$\cT\sA \subset \cT\cU$, respectively $\cT\sA' \subset \cT\cU'$, be the
triangulated subcategories generated by 
\[
\cO(-2m)\otimes p^*D^b\Vect(X), ..., \cO(-m-1)\otimes p^*D^b\Vect(X)
\]
and we let $\cT\cB$, resp.\,$\cT\cB'$, 
be the subcategory $\cO(-m)\otimes p^*D^b\Vect(X)$.
By Lemma \ref{lem:GWAddty}, $\GW(\cU,w,\vee)$ and $\GW(\cU',w,\vee)$ are both
equivalent to $\GW^{[r]}(X) \oplus \K(X)^{\oplus m}$.
In particular,
\begin{equation}
\label{eqn:P2m}
\GW^{[r]}(\P^{2m}_X;\cO(-2m)) \simeq \GW^{[r]}(X) \oplus \K(X)^{\oplus m}.
\end{equation}

When $n=2m$ is even,
we let $\cT\sA \subset \cT\cU$, respectively $\cT\sA' \subset \cT\cU'$, be
the triangulated subcategories generated by 
\[
\cO(-2m)\otimes p^*D^b\Vect(X),...,\cO(-m-1)\otimes p^*D^b\Vect(X),
\]
and $\cB=\cB'=0$.  In this case, $\GW(\cU,w,\vee)$ and $\GW(\cU',w,\vee)$ are
both equivalent to $\K(X)^{\oplus m}$, by Lemma \ref{lem:GWAddty}.
In particular, 
\begin{equation}
\label{eqn:P2m-1}
\GW^{[r]}(\P^{2m-1}_X;\cO(1-2m)) \simeq \K(X)^{\oplus m}.
\qedhere
\end{equation}
\end{proof}

As a result, we obtain a spectrum level version of 
some of Walter's calculations {\cite{WalterProj}}:

\begin{corollary}
\label{cor:GWProj}
Let $X$ be a scheme over $\Z[1/2]$ with an ample family of line bundles.
For all integers $r,n,i$ with $n\geq 0$, 
the Karoubi--Grothendieck--Witt spectrum
$\GW^{[r]}(\P^n_X;\cO(i))$
is equivalent to 
$$\renewcommand\arraystretch{1.5}
\begin{array}{ll}
\GW^{[r]}(X) \oplus \K(X)^{\oplus m}
& n=2m,\  i \text{ even,}\\
\K(X)^{\oplus m}
& n=2m-1,\  i \text{ odd,}\\
\GW^{[r]}(X) \oplus \K(X)^{\oplus m} \oplus \GW^{[r-n]}(X)
& n=2m+1,\  i \text{ even,}\\
\K(X)^{\oplus m} \oplus \GW^{[r-n]}(X)
& n=2m,\  i \text{ odd.}
\end{array}
$$
\end{corollary}

\begin{proof}
The line bundle $\cO(1)$ is canonically equipped with the non-degenerate
symmetric bilinear form $\cO(1)\otimes \cO(1) \stackrel{=}{\to} \cO(2)$ with
values in $\cO(2)$.  Cup product with that form induces equivalences of
Karoubi--Grothendieck--Witt spectra
\[
\GW^{[r]}(\P^n_X,\cO(i)) \simeq \GW^{[r]}(\P^n_X,\cO(i+2)).
\]
In particular, $\GW^{[r]}(\P^n_X,\cO(i))$ only depends on the parity of $i$.
Now, the first two computations of the corollary were already mentioned in the
proof of Theorem \ref{thm:Walter}; see (\ref{eqn:P2m}) and (\ref{eqn:P2m-1}).
The last two computations follow from those together with the statement of
Theorem \ref{thm:Walter}.
\end{proof}


\bigskip
\begin{thebibliography}{99}


\bibitem{Balmer-II}
P. Balmer,
Triangular Witt groups. II. From usual to derived.
{\em Math. Z.} 236 (2001), 351--382.

\bibitem{Balmer}
P. Balmer,
Witt cohomology, Mayer--Vietoris, homotopy invariance and
the Gersten conjecture, 
{\em $K$-theory} 23 (2001) 15---30.

\bibitem{BalmerGille} 
P. Balmer and S. Gille, 
Koszul complexes and symmetric forms over the punctured affine space,
{\em Proc. London Math. Soc.} 91 (2005), 273--299.

\bibitem{Gille} 
S. Gille, Homotopy invariance of coherent Witt groups,
{\em Math. Z.} 244 (2003), 211--233.

\bibitem{MKlocalisation} 
M. Karoubi, 
Localisation de formes quadratiques I, 
\textit{Ann.\ scient.\ \'Ec.\ Norm.\ Sup.\ (4)} 7 (1974), 359--404.

\bibitem{MKAnnalsH} 
M. Karoubi, 
Le th\'{e}or\`{e}me fondamental de la $K$-th\'{e}orie hermitienne,
\textit{Annals of Math.} 112 (1980), 259--282.

\bibitem{Knebusch}
M. Knebusch, 
Symmetric bilinear forms over algebraic varieties, 
pp. 103--283 in Queen's Papers in Pure Appl. Math. 46,
Queen's Univ., 1977.

\bibitem{Ranicki}
A. Ranicki,
Algebraic $L$-theory IV,
{\em Comment. Math. Helv.} 49 (1974) 137--167.

\bibitem{Rohrbach}
H. Rohrbach,
The projective bundle formula for Grothendieck-Witt spectra,
preprint, 2020.

\bibitem{SchlichtingMV}
M. Schlichting,
The Mayer-Vietoris principle for Grothendieck--Witt groups of schemes,
{\em Invent. Math.} 179 (2010), 349–-433.

\bibitem{Schlichting.Fund} 
M. Schlichting,
Hermitian $K$-theory, derived equivalences and Karoubi's fundamental theorem,
{\em J. Pure Appl. Alg.} 221 (2017), 1729--1844.

\bibitem{TT}
R. Thomason and T. Trobaugh,
Higher algebraic $K$-theory of schemes and of derived
categories, pp. 247--435 in
{\em The Grothendieck Festschrift III}
Progress in Math. 88, Birkh\"auser. 1990.

\bibitem{Walter}
C. Walter, 
Grothendieck--Witt groups of triangulated categories, preprint (2003),
available at http://www.math.uiuc.edu/K-theory/643

\bibitem{WalterProj}
C. Walter, 
Grothendieck--Witt groups of projective bundles, preprint (2003),
available at http://www.math.uiuc.edu/K-theory/644


\bibitem{W-MV}
C. Weibel,
Mayer-Vietoris sequences and module structures on $NK_*$ 
pp.~466--493 in 
Lecture Notes in Math.\,\#854, Springer-Verlag, 1981.

\bibitem{WK} 
C. Weibel, 
{\em The $K$-book}, 
Grad. Studies in Math. 145, AMS, 2013.


\end{thebibliography}
\end{document}